\documentclass{amsart}
\usepackage[all]{xy}
\usepackage{amsmath}
\usepackage{amsfonts}
\usepackage{amssymb}
\usepackage{amsthm}
\usepackage{amscd}
\usepackage{latexsym}
\usepackage{txfonts}
\usepackage{multirow}
\usepackage{float}
\usepackage{threeparttable}
\usepackage{enumerate}
\usepackage{rotating}
\usepackage{tabularx}
\usepackage{comment}
\usepackage{hyperref}
\usepackage{xcolor}
\usepackage{todonotes}

\theoremstyle{plain}
\newtheorem{theorem}{Theorem}[section]
\newtheorem{lemma}[theorem]{Lemma}
\newtheorem{cor}[theorem]{Corollary}

\newtheorem{definition}[theorem]{Definition}
\newtheorem{example}[theorem]{Example}
\newtheorem*{theorem*}{Theorem}
\theoremstyle{remark} 

\newcommand{\bb}{\mathbb}
\newcommand{\mrm}{\mathrm}
\newcommand{\mf}{\mathfrak}
\newcommand{\ov}{\overline}
\newcommand{\Sym}{\mathrm{Sym}}

\newcommand{\I}{\mathcal{S}}

\newcommand{\Ome}{\Omega}

\newcommand{\ch}{\mathbf{v}}

\def\s{\sigma}
\def\t{\tau}

\def\Irr{{\rm Irr}}

\begin{document}
\title{On the EKR-module Property}
\thanks{This work was partially supported  by NSFC grant 11931005 to the first author. The second author is a  PIMS CRG Post Doctoral Fellow.}
\date{}
\author{Cai Heng Li}
\address{SUSTech International Center for Mathematics\\
Department of Mathematics\\
Southern University of Science and Technology\\
Shenzhen, Guangdong 518055\\
P. R. China}
\email{lich@sustech.edu.cn}

\author{Venkata Raghu Tej Pantangi}
\address{Department of Mathematics\\
University of Lethbridge\\
Lethbridge, Alberta T1K3M4\\
Canada}
\email{pvrt1990@gmail.com}

\begin{abstract}
In the recent years, the generalization of the Erd\H{o}s-Ko-Rado (EKR) theorem to permutation groups has been of much interest. A transitive group is said to satisfy the \emph{EKR-module property} if the characteristic vector of every maximum intersecting set is a linear combination of the characteristic vectors of cosets of stabilizers of points. This generalization of the well-know permutation group version of the Erd\H{o}s-Ko-Rado (EKR) theorem, was introduced by K. Meagher in \cite{PSUEKR}. In this article, we present several infinite families of permutation groups satisfying the EKR-module property, which shows that permutation groups satisfying this property are quite diverse.
\end{abstract}

\maketitle

\section{Introduction}
The Erd\H{o}s-Ko-Rado (EKR) theorem \cite{EKR1961} is a classical result in extremal set theory. 
This celebrated result considers collections of pairwise intersecting $k$-subsets of an $n$-set. 
The result states that if $n\geq 2k$, for any collection $\mathfrak{S}$ of pairwise intersecting $k$-subsets, the cardinality $|\mathfrak{S}| \leq {n-1 \choose k-1}$. 
Moreover in the case $n>2k$, if $|\mathfrak{S}|={n-1 \choose k-1}$, then $\mathfrak{S}$ is a collection of $k$-subsets containing a common point. 
(When $n=2k$, the collection of $k$-subsets that avoid a fixed point, is also a collection of ${n-1 \choose k-1}$ pairwise intersecting subsets.) 
From a graph-theoretic point of view, this is the characterization of maximum independent sets in  Kneser graphs.

There are many interesting generalizations of this result to other classes of objects with respect to certain  form of intersection. 
One such generalization given by Frankl and Wilson \cite{FW1986} considers collections of pairwise non-trivially intersecting $k$-subspaces of a finite $n$-dimensional vector space, which corresponds to independent sets in $q$-Kneser graphs.
The book \cite{GMbook} is an excellent survey, including many generalizations of the EKR theorem.

In this article, we are concerned with EKR-type results for permutation groups. 
The first result of this kind was obtained by Deza and Frankl \cite{DF1977}, who investigated families of pairwise \emph{intersecting} permutations. 
Two permutations $\sigma, \tau \in S_{n}$ are said to {\it intersect} if the permutation $\sigma\tau^{-1}$ fixes a point. 
A set of permutations is called an {\it intersecting set} if $\s\t^{-1}$ fixes a point for any two members $\s$ and $\t$ of the set.
Clearly the stabilizer in $S_{n}$ of a point or its coset is a canonically occurring family of pairwise intersecting permutations, of size $(n-1)!$. In \cite{DF1977}, it was shown that if $\I$ is a family of pairwise intersecting permutations, then $|\I| \leq (n-1)!$. In the same paper, it was conjectured that if the equality $|\I|=(n-1)!$ is met, then $\I$ has to be a coset of a  point stabilizer. This conjecture was proved by Cameron and Ku (see \cite{CK2003}). An independent proof was given by Larose and Malvenuto (see \cite{LM2004}). Later, Godsil and Meagher (see \cite{GM2009}) gave a different proof. 
A natural next step is to ask similar questions about general transitive permutation groups.

Let $G$ be a finite group acting transitively on a set $\Omega$. 
An \emph{intersecting} subset of $G$ with respect to this action is a subset $\I\subset G$ in which any two elements intersect. 
Obviously, a point stabilizer $G_\alpha$, its left cosets $gG_\alpha$, and right cosets $G_\alpha g$ are intersecting sets, which we call \emph{canonical} intersecting sets.
An intersecting set of maximum possible size is called a \emph{maximum intersecting set}. 
Noting that the size of a canonical intersecting set is $|G_\alpha|=|G|/|\Omega|$, we see that the size of a maximum intersecting set is at least $|G|/ |\Omega|$. 
It is now natural to ask the following:

(A) Is the size of every intersecting set in $G$ bounded above by the size of a point stabilizer?

(B) Is every maximum intersecting set canonical?

As mentioned above, the results of Deza-Frankl, Cameron-Ku, and Larose-Malvenuto show that the answer to both these questions in positive for the natural action of a symmetric group. 
However, not all permutation groups satisfy similar properties, although there are many interesting examples that do. We now formally define the conditions mentioned in the above questions.

\begin{definition}
{\rm
A transitive group $G$ on $\Omega$ is said to satisfy the {\it EKR} property if every intersecting set has size at most $|G|/|\Omega|$, and further said to satisfy the {\it strict-EKR} property if every maximum intersecting set is canonical.
}
\end{definition}

When the action is apparent, these properties will be attributed to the group.
We have already seen that the natural action of $S_{n}$ satisfies both the EKR and the strict-EKR property. EKR properties of many specific permutation groups have been investigated (see \cite{AM2014,AM2015,LPSX2018, MSi2019, MS2011, MST2016}). 
In particular, it was shown that all $2$-transitive group actions satisfy the EKR property, see \cite[Theorem 1.1]{MSi2019}, but not every 2-transitive group satisfies the strict-EKR property; for instance, with respect to the $2$-transitive action of $\mrm{PGL}(n,q)$ (with $n\geq 3$) on the 1-spaces, the stabilizer of a hyperplane is also a maximum intersecting set.
%
%
However, it is shown in \cite{MSi2019} that $2$-transitive groups satisfy another interesting property called the EKR-module property, defined below.

For a transitive group $G$ on $\Omega$ and a subset $S\subset G$, let 
\[\ch_{S}=\sum\limits_{s \in S} s \in \bb{C}G,\]
%
the characteristic vector of $S$ in the group algebra $\bb{C}G$. 
For $\alpha, \beta \in \Omega$ and $g\in G$ with $g\cdot \alpha=\beta$, we write
\[\ch_{\alpha,\beta}=\sum_{\substack{ t \in G\\ t\cdot\alpha= \beta}} t=\sum\limits_{x\in G_\alpha}{gx}=g\sum\limits_{x\in G_\alpha}x=\ch_{gG_\alpha},\]
the characteristic vector of the canonical intersecting set $gG_\alpha$, which we call a {\it canonical vector} for convenience.
The next definition was first introduced in \cite{PSUEKR}. 

\begin{definition}
{\rm
A finite transitive group $G$ on a set $\Omega$ is said to satisfy the {\it EKR-module property} if the characteristic vector of each maximum intersecting set of $G$ on $\Ome$ is a linear combination of canonical vectors, that is, the vectors in $\{\ch_{\alpha,\beta}\mid \alpha,\beta \in \Ome\}$.
%
%
}
\end{definition}

The name is  from the so called ``module method'' described in \cite{ahmadi2015erdHos}. 
We remark that 
\begin{enumerate}[(a)]
\item a group action satisfying the strict-EKR property also satisfies the EKR-module property, but the converse statement is not true;
    
\item and there exist group actions that satisfy the EKR-module property but not the EKR property, see \cite[Theorem 5.2]{meagher2021triangles}, and the following Example~\ref{ex:A_4};

\end{enumerate}


\begin{example}\label{ex:A_4}
{\rm 
Consider the action of $A_{4}$ on the set $\Omega$ of cosets of a subgroup $H \cong \bb{Z}_{2}$. We observe that the Sylow 2-subgroup $N$ of $A_{4}$ is an intersecting set. As $4=|N|>|A_{4}/|\Omega|=2$, this action does not satisfy the EKR property.
Consider an intersecting set $\I$. Then given $t\in \I$, the set $\I t^{-1}$ is an intersecting set containing the identity. By the definition of an intersecting set, we have $\I t^{-1} \subset \bigcup\limits_{g \in A_{4}}gHg^{-1}=N$. This shows that any maximum intersecting set must be coset of $N$. Any coset of $N$ is a union of two disjoint cosets of $H$, and thus this action satisfies the EKR-module property.
}
\end{example}

We will now construct an example of a group action that satisfies the EKR property, but not the EKR-module property. Prior to doing so, we mention a well-known result. Consider a transitive action of a group $G$ on a set $\Ome$. A subset $R\subset G$ is said to be a \emph{regular subset}, if for any $(\alpha,\ \beta) \in \Ome^2$, there is a unique $r \in R$ such that $r\cdot \alpha =\beta$. Corollary 2.2 of \cite{AM2015} states that permutation groups which contain a regular subset, satisfy the EKR property.

For a group $G$ and a subgroup $H\leqslant G$, let 
\[[G:H]=\{xH\mid x\in G\},\]
the set of left cosets of $H$ in $G$.
Then $G$ acts transitively on $[G:H]$ by left multiplication.
Moreover, each transitive action of a group is equivalent to such a coset action.

\begin{example}\label{cex}
{\rm
Consider $G=S_{5}$ (isomorphic to $\mrm{PGL}(2,5)$) and a subgroup $H\leq G$ isomorphic to the dihedral group of size 12. We consider the action of $G$ on $\Omega=[G :H]$. We first show that this action satisfies the EKR property, by demonstrating the existence of a regular subset. We consider the cyclic subgroup $C:=\left\langle(1,2,3,4,5)\right\rangle$ and the $4$-cycle $t :=(2,3,5,4)$. We claim that $R := C \cup tC$ is a regular set. As $|R|=|\Omega|$, this claim will follow by showing that for $r,s \in R$ with $r\neq s$,  we have $rgH \neq sgH$, for all $g \in G$. This is equivalent to showing that $g^{-1}r^{-1}sg \notin H$, for all $g \in G$. It is easy to verify that for any two distinct $r,s\in R$ , $r^{-1}s$ is either a $4$-cycle or a $5$-cycle, and thus $g^{-1}r^{-1}sg \notin H$. We can now conclude that $R$ is a regular subset. Therefore, by \cite[Corollary 2.2]{AM2015},  $G$ satisfies the EKR property. Thus the size of any intersecting set is bounded above by $|H|=12$. 
 
 Now we consider subgroup $K \cong A_{4}$ of $G$. It is easy to check that $KK^{-1}=K \subset \cup gHg^{-1}$ and thus $K$ is a maximum intersecting set. In this case, canonical intersecting sets are cosets of a conjugate of $H$. Every canonical intersecting set contains exactly $3$ even permutations. Also every permutation in $K$ is even. Now consider the sign character $\lambda$. For every canonical intersecting set $\I$, we have $\lambda(\ch_{S})=0$. On the other hand, we have $\lambda(\ch_{K})\neq 0$. From this, we see that $\ch_{K}$ cannot be a linear combination of the characteristic vectors of the canonical intersecting sets. Therefore this action does not satisfy the EKR-module property.
 }
\end{example}

We will now describe the main results of our paper.

\subsection{Main Results}
Our first result is a characterization of the EKR-module property of a group action, in terms of the characters of the group in question. Given a group $G$, a complex character $\chi$ of $G$, and a subset $A \subseteq G$, by $\chi(\ch_{A})$, we denote the sum $\sum\limits_{a\in A}\chi(a)$. We now describe our first result, a characterization of the EKR-module property in terms of character sums.

\begin{theorem}\label{idekrm}
Let $G$ be a finite group, $H<G$, and $\Omega=[G:H]$. 
Let $C=\{\chi \in \Irr(G):\ \chi(\ch_{H})=0\}$, and $\mf{S}$ be the collection of maximum intersecting sets in $G$.
Then $G$ on $\Ome$ satisfies the EKR-module property if and only if $\chi(\ch_{\I}) =0$ for any $\I \in \mf{S}$ and any $\chi \in C$.
\end{theorem}

We note that Example~\ref{cex} can be viewed as an application of the above result. In Example~\ref{cex}, the sign character $\lambda$ is a character such that $\lambda(\ch_{H})=0$. However, for the maximum intersecting set $K$, we have $\lambda(\ch_{K})\neq 0$. Thus by Theorem~\ref{idekrm}, the action in Example~\ref{cex} does not satisfy the EKR-module property. 

Given an action of $G$ on $\Omega$,  the derangement graph $\Gamma_{G,\ \Omega}$ is the graph whose vertex set is $G$, and vertices $g,h\in G$ are adjacent if and only if $gh^{-1}$ does not fix any point in $\Omega$. 
Then a set $S\subseteq G$ is an intersecting set if and only if it is an independent set in $\Gamma_{G,\ \Omega}$.
Therefore, the study of intersecting sets could benefit from the various results from spectral graph theory about independent sets. 
Many authors (for instance, see \cite{MS2011}, \cite{MST2016}) have studied the EKR and strict-EKR properties of various group actions from this point of view. 
Theorem~\ref{Specsuff} is a characterization of the EKR-module property of a group action in terms of spectra of weighted adjacency matrices of the corresponding derangement graph. 

It is well known (see \cite[Corollary 2.2]{AM2015}) that permutation groups with regular subsets satisfy the EKR property. However, as observed in example~\ref{cex}, such groups do not necessarily satisfy the EKR-module property. The following theorem shows that every permutation group with a regular normal subgroup satisfies the EKR-module property.

\begin{theorem}\label{rega}
Transitive groups actions with a regular normal subgroup satisfy the EKR-module property.
\end{theorem}

A few classes of permutation groups with a regular normal subgroup are Frobenius groups, affine groups, primitive groups of type HS, HC, and TW (for a description of these, we refer the reader to \cite{praeger1996finite}).

After showing that all $2$-transitive groups satisfy the EKR-module property, the authors of \cite{MSi2019} mention that the next natural step is to consider rank $3$ permutation groups. As a first step, we consider this problem for the class of primitive rank $3$ permutation groups. The next theorem reduces the problem to almost simple groups.
(Recall that a finite group is called {\it almost simple} if it has a unique minimal normal subgroup, which is non-abelian and simple.)

\begin{theorem}\label{Rank3}
Let $G$ be a primitive permutation group on $\Ome$ of rank $3$.
Then either $G$ has the EKR-module property, or $G$ is an almost simple group.
%
\end{theorem}

%
%

We would like to mention that when $n$ is sufficiently large, the rank $3$ action of $S_{n}$ on $2$-subsets of $[n]$, satisfies the strict-EKR property, and thereby the EKR-module property. This was proved in \cite{ellis2012setwise}. Example~\ref{cex} shows that this fails when $n=5$. 

In \cite{BMMKK2015}, a finite group $G$ is defined to satisfy the {\it weak EKR property}, if \textbf{every transitive action} of $G$ satisfies the EKR property. A finite group $G$ is defined to satisfy the {\it strong EKR property}, if \textbf{every transitive action} of $G$ satisfies the strict-EKR property. Theorem~$1$ of \cite{BMMKK2015} shows that nilpotent group satisfies the weak EKR property.  
This result was extended to supersolvable groups in \cite{LSP}. 
It is easy to check that every abelian group satisfies the strong EKR property. 
Theorem~$3$ of \cite{BMMKK2015} states that a finite non-abelian nilpotent group satisfies the strong EKR property if and only if it is a direct product of a $2$-group and an abelian group of odd order. 
We now make the following analogous definition.

\begin{definition}\label{sekrm}
{\rm
A finite group $G$ is said to satisfy the {\it EKR-module property} if every transitive action of $G$ satisfies the EKR-module property.
}
\end{definition}

As mentioned before, groups with the strict-EKR property have the EKR-module property. 
It is then natural to ask whether the converse statement is true or not.
It is shown in \cite[Theorem\,3]{BMMKK2015} that there are infinitely many nilpotent groups of nilpotency class $2$ that do not satisfy the strict-EKR property.
The next result then answers the question in negative.

\begin{theorem}\label{Nil}
Nilpotent groups of nilpotency class $2$ satisfy the EKR-module property.
\end{theorem}

It is shown in \cite[Theorem~2]{BMMKK2015} that a group $G$ satisfying the EKR property for every transitive action is necessarily solvable.
However, it is shown in Lemma~\ref{A5} that there do exist non-solvable groups which have the EKR-module property.

An analogue of the EKR-module property has been observed in other generalizations of the EKR theorem. Consider a graph $X$ and a prescribed set of ``canonically'' occurring cliques. We say that the graph satisfies the EKR-module property, if the characteristic vector of any maximum clique is a linear combination of the characteristic vectors of the canonical cliques. In the context of permutation groups satisfying the EKR-module property, the complement of the corresponding derangement graph satisfies the EKR-module property. In Chapter~5 of \cite{GMbook}, there are a few examples of strongly regular graphs satisfying the EKR-module property. In \cite{EKRpaley}, the authors show that Peisert-type graphs satisfy the EKR-module property. Let $q$ be an odd prime power. Let $F$ and $E$ be finite fields of order $q^{2}$ and $q$ respectively. A Peisert-type graph of type $(m,q)$ is a Cayley graph of the form $Cay(F,S)$, where the ``connection'' set $S$ is a union of $m$ distinct cosets of the multiplicative group $E^{\times}$ in $F^{\times}$. It is clear that any set of the form $sE+b$, with $s \in S$ and $b\in F$, is a clique. We deem these to be the canonical cliques. In \cite{EKRpaley}, the authors show that characteristic vector of any maximum clique in a Peisert-type graph, is a linear combination of the characteristic vectors of canonical cliques. In \S~\ref{srg}, we give a shorter independent proof of the same.

\section{EKR-module property and character theory.}\label{prelim}

In this section, we gather some tools which are used to prove our main results, and then prove Theorem~\ref{idekrm}.


Let $K=G_{(\Omega)}$ be the kernel of $G$ on $\Omega$. Here $\Omega=[G : H]$ for some $H \leq G$.
The following simple lemma shows that we may assume without loss of generality that $K$ is trivial. 

\begin{lemma}\label{lem:kernel}
Let $\pi:G \to G/K$ be the natural quotient map. 
Then a subset $\I\subset G$ is a maximum intersecting set of $G$ if and only if $\pi(\I)$ is a maximum intersecting set of $G/K$.
%
\end{lemma}
\begin{proof}
Given an intersecting set $A \subset G$, we note that $AK:=\{ak\ :\ a\in A\ \&\ k\in K\}$ is also an intersecting set. So any maximum intersecting set in $G$ must be a union of $K$-cosets. Let $s_{1},s_{2},\ldots s_{r} \in G$ be such that $\I = \bigcup s_{i}K$ is a maximum intersecting subset of $G$. We see that $\pi(\I)=\{s_{i}K\ :\ 1\leq i \leq r\}$ is an intersecting set in $G/K$.

Now consider a maximum intersecting set $\mathcal{T}=\{t_{i}K\ :\ 1\leq i\leq s\}$ of $G/K$. It is clear that $\pi^{-1}(\mathcal{T})=\bigcup t_{i}K \subset G$ is an intersecting set. So we have $|\mathcal{T}||K|=|\pi^{-1}(\mathcal{T})|\leq |\I|$, and $|\I|/|K|=|\pi(\I)| \leq |\mathcal{T}|$. This shows that $\pi(\I)$ (respectively $\pi^{-1}(\mathcal{T})$) is a maximum intersecting set of $G/K$ (respectively $G$).
\end{proof}
As an immediate consequence, we get the following corollary.

\begin{cor}\label{permred}
Let $G$ be a finite transitive group on $\Omega$ with kernel $K=G_{(\Omega)}$, and let $\pi:G \to G/K$ be the natural quotient map. 
Then the following hold:
\begin{enumerate}[{\rm(i)}]
\item $G$ satisfies the EKR (respectively strict-EKR) property if and only if $G/K$ satisfies the EKR (respectively strict-EKR);

\item $G$ satisfies the EKR-module property if and only if $G/K$ satisfies the EKR-module property.
\end{enumerate}
\end{cor}
\begin{proof}
Set $Q:=G/K$. We note that for all $\omega\in \Omega$ and $g\in G$, we have $\pi(g)Q_{\omega}=\pi(gG_{\omega})$ and $gG_{\omega}=\pi^{-1}(\pi(g)Q_{\omega})$. 
The proof now follows from Lemma~\ref{lem:kernel}.
\end{proof}

For any $g\in G$, we denote by $H^{g}$ the subgroup $gHg^{-1}$.
Given $\alpha=aH \in \Omega$, we have $G_{\alpha}=H^{a}$. Thus, with respect to this action, we see that the set $\{aH^{b}\ :\ a,b\in G\}$ is the set of canonical intersecting sets. By $\mf{I}_{G}(\Omega)$, we denote the subspace of $\bb{C}G$ spanned by the set $\{\ch_{aH^{b}}\ :\ a,b\in G\}$ of the characteristic vectors of the canonical intersecting sets. By the definition of the EKR-module property, the action of $G$ on $\Omega$ satisfies the EKR-module property if and only if $\ch_{\I} \in \mf{I}_{G}(\Omega)$ for every maximum intersecting set $\I$ in $G$.

We observe that for every $a,\ b,\ g,\ h \in G$, we have $g\ch_{aH^{b}}h= \ch_{gah H^{h^{-1}b}}$. Therefore, $\mf{I}_{G}(\Omega)$ is a two-sided ideal of the group algebra $\bb{C}G$. The two-sided ideals of complex group algebras are characterized by the Artin-Wedderburn decomposition.

We will now recall some basic facts on group algebra, proofs of which can be found in any standard text on representation theory such as \cite{isaacs1994character}.
Let $\Irr(G)$ be the set of irreducible complex characters of $G$. For $\chi \in \Irr(G)$, we define $M_{\chi}:=\sum\limits_{i=1}^{\chi(1)}W_{i}$, where $\{W_{1},W_{2}, \ldots W_{\chi(1)}\}$ are the right submodules of $\bb{C}G$ that afford the character $\chi$. By Maschke's theorem, we have the decomposition

$$\bb{C}G=\bigoplus\limits_{\chi \in \Irr(G)} M_{\chi}.$$

 For each $\chi \in \Irr(G)$, we have $\mrm{dim}_{\bb{C}}(M_{\chi})=\chi(1)^{2}$ and that $M_{\chi}$ is a minimal two-sided ideal containing the primitive central idempotent
$$e_{\chi}=\dfrac{\chi(1)}{|G|}\sum\limits_{g \in G} \chi(g^{-1})g.$$ By orthogonality relations among characters, we have
\begin{equation}\label{oidem}
e_{\chi}e_{\psi} =\begin{cases} 1\ \text{,\ \ if $\chi=\psi$,} \\
0\ \text{,\ \ otherwise.}
\end{cases}
\end{equation}

Using the fact that $\bb{C}G$ is a semi-simple algebra, we now get the following description of two-sided ideals of $\bb{C}G$

\begin{lemma}\label{ides}
Given a two-sided ideal $\mathfrak{J}$ of $\bb{C}G$, there is a subset $Y_{\mathfrak{J}} \subseteq \Irr(G)$ such that $\mathfrak{J}=\bigoplus\limits_{\chi \in Y_{\mathfrak{J}}} \left\langle e_{\chi}\right\rangle$.
\end{lemma}

Our investigation of the EKR-module property of the action of $G$ on $\Omega$, will benefit from the description of $\mf{I}_{G}(\Omega)$ as a direct sum of simple ideals of $\bb{C}G$. We recall that $\mf{I}_{G}(\Omega)$, is the subspace of $\bb{C}G$ spanned by the set $\{\ch_{aH^{b}}\ :\ a,b\in G\}$. 
We also showed that it is a two-sided ideal.

\begin{lemma}\label{caniddes}
Let $G$ be a finite group, $H<G$ a subgroup, and $\Omega=[G : H]$ be the space of left cosets of $H$. Let $Y_{H}=\{\chi \in \Irr(G)\ :\ \chi(\ch_{H}) \neq 0\}$.
Then, the linear span $\mf{I}_{G}(\Omega)$ of $\{\ch_{aH^{g}}\ :\ a,g\in G\}$ decomposes as the sum $\bigoplus\limits_{\chi \in Y_{H}} \left\langle e_{\chi}\right\rangle$ of simple ideals of $\bb{C}G$.
\end{lemma}
\begin{proof}
For any subset $S \subset G$ and $\psi \in \Irr(G)$, we have
\begin{align*}
\frac{|G|}{\psi(1)}e_{\psi}\ch_{S}&= \frac{|G|}{\psi(1)}\sum\limits_{s \in S} e_{\psi}s\notag\\
&=\sum\limits_{s \in S}\sum\limits_{g \in G} \psi(sg^{-1})g \notag \\
&= \sum\limits_{g\in G}g \sum\limits_{s \in S} \psi(sg^{-1})\notag\\
&=\sum\limits_{g\in G} \psi(\ch_{S g^{-1}})g.
\end{align*}
Therefore for any $\chi \in Y_{H}$, we have $0 \neq e_{\chi} \ch_{H} \in \left\langle e_{\chi}\right\rangle \cap \mf{I}_{G}(\Omega) \subset \mf{I}_{G}(\Omega)$. As $\left\langle e_{\chi}\right\rangle$ is a minimal ideal, we conclude that $\left\langle e_{\chi}\right\rangle \subset \mf{I}_{G}(\Omega)$.

Now consider $\theta \in \Irr(G) \setminus Y_{H}$. In this case, we have $\theta(\ch_{H^g})=\theta(\ch_{H})=0$. Let $\Theta : G \to \mrm{GL}_{\theta(1)}(\bb{C})$ be a \textbf{unitary} representation affording $\theta$ as its character. Given a subset $S\subset G$, we define $M_{S}:=\sum\limits_{s \in S}\Theta(s)$. As $\Theta$ is a unitary representation, we have $M_{S^{-1}}= M_{S}^{\dagger}$, that is, $M_{S^{-1}}$ is the conjugate transpose of $M_{S}$. Now given $a, g \in H$, we have
$$M_{H^{g}a^{-1}}M_{aH^{g}}= \sum \limits_{x, y \in H^{g}} \Theta(x)\Theta(y)= \sum \limits_{x, y \in H^{g}} \Theta(xy)= |H|M_{H^{g}}.$$ 
As $\theta(\ch_{H^{g}})=0$, we have $0=\mrm{Tr}(M_{aH^{g}})=\mrm{Tr}(M_{H^{g}a^{-1}}M_{aH^{g}})$. 
Since $M_{H^{g}a^{-1}}=M_{aH^{g}}^{\dagger}$, this can only happen when $M_{aH^{g}}=0$. 
Therefore, $\theta(\ch_{aH^{g}})=0$, and we conclude that $e_{\theta}\ch_{aH^{g}}=0$. 
Thus $e_{\theta}$ annihilates $\mf{I}_{G}(\Omega)$.
As $e_{\theta}^{2}=1\neq 0$, $e_{\theta}$ cannot be an element of the ideal $\mf{I}_{G}(\Omega)$. Now the result follows by applying Lemma \ref{ides} and equation \eqref{oidem}.
\end{proof}

As an immediate application, we obtain a significantly shorter proof of Lemma 4.1 and Lemma 4.2 of \cite{ahmadi2015erdHos}. The content of these two results is presented as the following corollary. We will use the following technical result in the proof of Theorem~\ref{Rank3}. We would like to mention that it was a key result that led to the ``Module Method'' described in \cite{ahmadi2015erdHos}.

\begin{cor}\label{2t}
Let $G \leq \Sym(\Omega)$ be a $2$-transitive permutation group with $\pi \in \Irr(G)$ such that $1+\pi$ is the corresponding permutation character. Given $\alpha, \beta \in \Omega$, set $\ch_{\alpha, \beta} := \sum\limits_{\{g \in G\ :\ g\cdot \alpha =\beta\}}g$ and $\ch_{G}:=\sum\limits_{g\in  G}g$.
 Then
\begin{enumerate}
\item $\mf{I}_{G}(\Omega)=\left\langle e_{1}\right\rangle + \left\langle e_{\pi}\right\rangle $ is a vector space of dimension $1+(|\Omega|-1)^{2}$;
\item and the set $$B_{\omega}:=\{\ch_{G}\} \cup \ \{\ch_{\alpha,\ \beta}\ : \ (\alpha,\beta)\in \left(\Omega \setminus \{\omega\} \right)^{2}\},$$ is a basis set of $\mf{I}_{G}(\Omega)$ for any $\omega \in \Omega$.
\end{enumerate}
\end{cor}
\begin{proof}
Part (1) follows immediately from Lemma~\ref{caniddes}.

We observe that for $\alpha \in \Omega$, we have $\ch_{\alpha, \omega}=\ch_{G}-\sum\limits_{\beta \neq \omega}\ch_{\alpha,\ \beta}$. 
Therefore, every vector of the form $\ch_{\gamma,\ \delta}$ is in the linear span of the elements of $B_{\omega}$, and thus $B_{\omega}$ spans $\mf{I}_{G}(\Omega)$. 
Linear independence follows as $\mf{I}_{G}(\Omega)$ is a $|B_{\omega}|$-dimensional subspace.
\end{proof}

We are now ready to prove Theorem~\ref{idekrm}.

{\bf Proof of Theorem~\ref{idekrm}:}
Given $\chi \in C=\{\chi \in \Irr(G)\ :\ \chi(\ch_{H})=0\}$, by Lemma~\ref{caniddes} and \eqref{oidem}, we have $e_{\chi}x=0$, for all $x \in \mf{I}_{G}(\Ome)$. By the definition of the EKR-module property, for any maximum intersecting set $\I$, we have $\ch_{\I}\in \mf{I}_{G}(\Ome)$. 
The equality $\chi(\ch_{\I})=0$ follows from $e_{\chi} \ch_{\I}=0$.

We now prove the other direction. Suppose that for any $\chi \in C$ and any maximum intersecting set $\I$, we have $\chi(\ch_{\I})=0$. Fix a maximum intersecting set $\I$. If $\I$ is a maximum intersecting set, then so is $\I g^{-1}$, for all $g\in G$. Therefore, $\chi(\ch_{\I g^{-1}})=0$, for all $\chi \in C$ and all $g \in G$. Thus $\frac{|G|}{\chi(1)}e_{\chi}\ch_{\I}= \sum\limits_{g \in G} \chi(\ch_{\I g^{-1}})g=0$, for all $\chi \in C$. 
Further, by the equality $\sum\limits_{\psi \in \Irr(G)}e_{\psi}=1$, we have
\begin{align*}
\ch_{\I} & = \left(\sum\limits_{\psi \in \Irr(G)}e_{\psi}\right) \times \ch_{\I}
= \left(\sum\limits_{\psi \notin C}e_{\psi}\right) \times \ch_{\I}.
\end{align*}
By Lemma~\ref{caniddes}, $\sum\limits_{\psi \notin C}e_{\psi} \in \mf{I}_{G}(\Ome)$. Since $\mf{I}_{G}(\Ome)$ is an ideal, we have $\ch_{\I} \in \mf{I}_{G}(\Ome)$. 
Thus the EKR-module property is satisfied.
\qed

We note that if $\I$ is a maximum intersecting set, then for any $t\in \I$, the set $\I t^{-1}$ is a maximum intersecting set that contains the identity element. So every maximum intersecting set is a ``translate'' of an intersecting set containing the identity. The following corollary shows that, as far as the EKR-module property is concerned, we can restrict ourselves to maximum intersecting sets containing the identity.

\begin{cor}\label{idekr}
Let $G$ be a finite group with the identity $1_{G}$, $H<G$, and $\Omega=[G:H]$. 
Let $C=\{\chi \in \Irr(G): \ \chi(\ch_{H})=0\}$, and 
$$\mf{S}_{0}=\{\I_{0}\ :\ \I_{0}\  \text{is a maximum intersecting set with $1_{G}\in \I_{0}$}\}.$$
Then $G$ on $\Ome$ satisfies the EKR-module property if and only if $\chi(\ch_{\I_{0}}) =0$ for any $\I_{0} \in \mf{S}_{0}$ and any $\chi \in C$.
\end{cor}
\begin{proof}

At first, we assume that $\chi(\ch_{\I_{0}})=0$, for all $\I_{0} \in \mf{S}$ and $\chi \in C$.
Fix a $\chi \in C$ and a maximum intersecting set $\I$ . Let $P:G \to \mrm{GL}_{n}(\bb{C})$ be a unitary representation affording $\chi$ as its character. Given a set $X \subset G$, define $M_{X}:=\sum\limits_{x \in X} P(x)$. We observe that $M_{\I}M_{\I^{-1}}=\sum\limits_{t\in \I}M_{\I t^{-1}}$. 
Then $Tr(M_{\I t^{-1}})=\chi(\ch_{\I t^{-1}})$. As $P$ is a unitary representation, $M_{\I^{-1}}$ is the conjugate transpose of $M_{\I}$, and thus
\begin{equation}\label{unitary}
Tr\left(M_{\I}M_{\I}^{\dagger}\right)=\sum\limits_{t\in \I}\chi(\ch_{\I t^{-1}}).
\end{equation}
For any $t\in \I$, the set $\I t^{-1}$ is a maximum intersecting set containing the identity $1_{G}$. Therefore, we have $\chi(\ch_{\I t^{-1}})=\chi(\ch_{H})=0$, for all $t \in \I$. Thus by \eqref{unitary}, we have $Tr(M_{\I}M_{\I}^{\dagger})=0$. As $M_{\I}^{\dagger}$ is the conjugate transpose of $M_{\I}$, the matrix $M_{\I}M_{\I}^{\dagger}$ is a diagonal matrix whose entries are the norms of rows of $M_{\I}$. Thus,
$Tr(M_{\I}M_{\I}^{\dagger})=0$ implies that $M_{\I}=0$. We can now conclude that $\chi(\ch_{\I})=Tr(M_{\I})=0$. By Theorem~\ref{idekrm}, the action of $G$ on $\Omega$ satisfies the EKR-module property.

The other direction follows directly from Theorem~\ref{idekrm}.
\end{proof}

\section{EKR-module property and Spectral graph theory.}\label{prelim2}

Results from spectral graph theory have proved useful in characterizing maximum intersecting sets in some permutation groups (for instance, see \cite{MSi2019}, \cite{MS2011}, \cite{MST2016}). 
Let $G$ be a group acting on $\Omega=[G:H]$, for some $H \leq G$.
An element $g\in G$ is called a derangement if it does not fix any point in $\Omega$.  
Let $Der(G,\Omega)$ denote the set of derangements in $G$. 
It is easy to see that $Der(G,\Omega)=G \setminus \bigcup\limits_{g \in G} gHg^{-1}$. By $\Gamma_{G,\Omega}$, we denote the Cayley graph on $G$, with $Der(G, \Omega)$ as the ``connection set''. We now observe that intersecting sets in $G$ are the same as independent sets/co-cliques in $\Gamma_{G,\Omega}$. This observation enables us to use some popular spectral bounds on sizes of independent sets in regular graphs. Before describing these, we recall some standard definitions.

For graph $X$ on $n$ vertices, a real symmetric matrix $M$ whose rows and columns are indexed by the vertex set of $X$, is said to be \emph{compatible} with $X$, if $M_{u,v}=0$ whenever $u$ is not adjacent to $v$ in $X$. Clearly, the adjacency matrix of $X$ is compatible with $X$. 
Given a subset $S$ of the vertex set, by $\ch_{S}$, we denote the characteristic vector of $S$. 
We now state the following famous result which is referred to as either the Delsarte-Hoffman bound or the ratio bound.

\begin{lemma}(\cite[Theorem 2.4.2]{GMbook})\label{wrb}
Let $M$ be a real symmetric matrix with constant row sum $d$, which is compatible with a graph $X$ on $n$ vertices. If the least eigenvalue of $M$ is $\tau$, then for any independent set $S$ in $X$,
$$|S| \leq \frac{n(-\tau)}{d-\tau},$$ and if equality holds, then
$$\ch_{S}-\frac{|S|}{n}\ch_{X}$$ is a $\tau$-eigenvector for $M$.
\end{lemma}

The application of the above lemma on clever choices of $\Gamma_{G,\Omega}$-compatible matrices, proved useful in characterization of maximum intersecting sets for many permutation groups (for instance see \cite{MS2011} and \cite{MST2016}). We will now describe these in detail.

\begin{definition}
{\rm
Let $G$ be a group acting transitively on a set $\Omega$. A $(G,\Omega)$-compatible class function is a real valued class function $f: G \to \bb{R}$ such that: (i) $f(g)=0$ for all $g \notin Der(G,\Omega)$; and (ii) $f(d)=f(d^{-1})$ for all $d\in Der(G, \Omega)$.
}
\end{definition}

Let $f$ be a $(G,\Omega)$-compatible class function . Consider the matrix $M^{f}$ indexed by $G\times G$, that satisfies $M^{f}_{g,h}=f(gh^{-1})$ for all $(g,h)\in G \times G$. Clearly $M^{f}$ is a $\Gamma_{G,\Omega}$-compatible matrix. We now describe the spectra of such matrices.
The description of spectra of matrices of the form $M^{f}$ is a special case of well-know results by Babai (\cite{Babai}) and Diaconis-Shahshahani (\cite{DS1981}). The following lemma, which is a special case of Lemma~5 of \cite{DS1981}, describes the spectra of matrices of the form $M^{f}$.

\begin{lemma}{\rm(Babai, Diaconis-Shahshahani)}\label{spec}
Let $G$ be a permutation group on $\Omega$, with $Der(G,\Ome) \subset G$ being the set of derangements. Let $f: G \to \bb{R}$ be $(G,\Omega)$-compatible class function. Define $M^{f} \in \bb{C}^{G\times G}$ to be the matrix satisfying $M^{f}_{g,\ h}=f(g^{-1}h)$, for all $g,h \in G$. Then $M^{f}$ is a $\Gamma_{G,\Omega}=Cay(G,Der(G,\Ome))$-compatible matrix with spectrum
$\mrm{Spec}(M^{f}):=\{\lambda_{\chi,f}\ :\ \chi \in \Irr(G)\}$, where $$\lambda_{\chi,f}=\frac{1}{\chi(1)}\sum\limits_{g \in G}f(g)\chi(g).$$ Given $\nu \in \mrm{Spec}(M^{f})$, the $\nu$-eigenspace in $\bb{C}{G}$ is the two-sided ideal
$$\sum\limits_{\{\chi \ :\ \chi\in \Irr(G)\ \text{and}\ \lambda_{\chi,f}=\nu\}} \left\langle e_{\chi}\right\rangle.$$
\end{lemma}

We are now ready to give a sufficient condition for EKR-module property in terms of spectra of $\Gamma_{G,\Omega}$-compatible matrices. 
Let $f: G \to \bb{R}$ be a $(G, \Omega)$-compatible class function. 
Then the row sum of $M^{f}$ is $r_{f}:=\sum\limits_{g\in G}f(g)$. Let $\tau$ be the least eigenvalue of $M^{f}$. 
By Lemma~\ref{wrb}, for any intersecting set $S$, we have 
$$|S|\leqslant \frac{|G|(-\tau)}{r_{f}-\tau}.$$ 
Let us assume that equality holds for some intersecting set $\I$.
By Lemmas~\ref{wrb} and \ref{spec}, if $\I$ is an maximum intersecting set, then
$\ch_{\I}$ is in the $2$-sided ideal
$$\left\langle e_{1} \right\rangle +\sum\limits_{\{\chi\ :\ \chi\in \Irr(G)\ \text{and}\ \lambda_{\chi,f}=\tau\}} \left\langle e_{\chi}\right\rangle.$$ 

Now by application of Lemma~\ref{caniddes}, we obtain the following sufficient condition for EKR-module property.

\begin{theorem}\label{Specsuff}
Let $G$ be a group acting on the set $\Omega$ of left cosets of a subgroup $H$. Assume that there is an intersecting set $S$ and a $(G,\Omega)$-compatible class function $f:G \to \bb{R}$ such that $|S|=\dfrac{|G|(-\tau)}{d-\tau}$, where $d=\sum\limits_{g \in G} f(g)$ and $\tau$ is the least eigenvalue of $M^{f}$.
Then

(a) $|S|$ is the size of a maximum intersecting set in $G$; and

(b) the action of $G$ on $\Omega$ satisfies the EKR-module property if
$$\left\{\chi \in \Irr(G)\ :\ \frac{1}{\chi(1)}\sum\limits_{g \in G}f(g)\chi(g)=\tau\right\}\subseteq \left\{\chi \in \Irr(G)\ :\ \chi(\ch_{H}) \neq 0\right\}.$$

\end{theorem}

At this point, we remark that the proofs of EKR (\cite{MST2016}) and EKR-module properties (\cite{MSi2019}) of $2$-transitive groups, involved finding a class function that satisfies the conditions of Theorem~\ref{Specsuff}.

\section{Proof of Theorem~\ref{rega}}\label{regular}

In this section, we prove Theorem~\ref{rega}.
By Corollary~\ref{permred}, we can restrict ourselves to permutation groups that contain a regular normal subgroup. Let $A$ be a finite group and $H \leqslant Aut(A)$. 
We consider the permutation action of $G:= A \rtimes H$ on $A$, defined by $(a,\sigma)\cdot b=a\sigma(b)$, for all $a,b\in A$ and $\sigma \in H$. 
It is well-known that any permutation group with a regular normal subgroup, is of the form $G\leqslant \mrm{Sym}(A)$. 
By \cite[Corollary 2.2]{AM2014}, permutation groups which contain a regular subgroup, satisfy the EKR property. Thus the action of $G$ on $A$ satisfies the EKR property.

Before starting the proof, we prove an elementary result that we will use later. 
Every element of $G$ is of the form $(a, \sigma)$, where $a\in A$ and $\sigma \in H$. 
Note that $(a,\sigma)(b,\pi)= (a\sigma(b), \sigma\pi)$. 
We need the following well-known result for technical reasons.

\begin{lemma}\label{semidirect}
Consider $g=(a,\sigma) \in G$, with $a\in A$ and $\sigma \in H$. If $g$ fixes a point then

(i) $g$ is conjugate to $\sigma$ via an element of $A$; and

(ii) $\sigma$ is the unique $A$-conjugate of $g$ in $H$.
\end{lemma}
\begin{proof}
For convenience, given $x\in A$, we identify $(x,\ 1_{H}) \in G$ with $x \in A$.
Given any $b \in A$, we have $(a,\sigma)\cdot b=a\sigma(b)$. 
So $(a,\sigma)$ fixes $b$ if and only if $a=b\sigma(b^{-1})$.

Now for $c\in A$, we have $c^{-1}(b\sigma(b^{-1}),\sigma)(c)=(c^{-1}b\sigma(bc^{-1}),\sigma)$. 
Thus$(c^{-1}b\sigma(bc^{-1}),\sigma) \in H$ if and only if $c^{-1}b\sigma(bc^{-1})=1_{G}$. 
Then the proof follows from setting $c=b$.
\end{proof}

We will now prove the theorem by using Corollary~\ref{idekr}.
Let $\I_{0}$ be any maximum intersecting set with $1_{G} \in \I_{0}$. 
As $\I_{0}$ is an intersecting set, for all $s \in \I_{0}$, the element $s=s1^{-1}_{G}$ fixes some point. 
Thus by Lemma~\ref{semidirect}, given $s\in \I_{0}$, there exists a unique element $\sigma_{s}\in H$ and an element $a_{s} \in A$, such that $a_{s}^{-1}s a_{s}=\sigma_{s}\in H$. 
We now claim that $\{\sigma_{s}\ :\ s\in \I_{0}\}=H$.
Since $G$ satisfies the EKR property, we have $|\I_{0}|=|H|$. 
Therefore, $\{\sigma_{s}:\ s\in \I_{0}\}=H$ is equivalent to injectivity of the map $s \mapsto \sigma_{s}$.

Suppose that for some $s,r \in \I_{0}$, we have $\sigma_{s}=\sigma_{r}$. Then, we have \[sr^{-1}=a_{s}\sigma_{s}a_{s}^{-1}a_{r}\sigma_{r}^{-1}a_{r}^{-1}= a_{s}\sigma_{s}(a_{s}^{-1}a_{r})\sigma_{s}^{-1}a_{r}^{-1}\in A.\]
As $\I_{0}$ is an intersecting set, $sr^{-1}\in A$ fixes a point. Since $A$ acts regularly, we must have $sr^{-1}=1$. 
Thus $s \mapsto \sigma_{s}$ is injective, and $\{\sigma_{s}:\ s\in \I_{0}\}=H$. 
As $s \in \I_{0}$ is conjugate to $\sigma_{s}$, for any $\psi \in \Irr(G)$, we have $\psi(\ch_{\I_{0}})=\sum\limits_{s\in \I_{0}} \psi(\sigma_{s})=\psi(\ch_{H})$.
Therefore, if $\chi \in \{\psi:\ \psi\in \Irr(G)\ \&\ \psi(\ch_{H})=0\}$ and $\I_{0}$ is a maximum intersecting set containing $1_{G}$, we have $\chi(\ch_{\I_{0}})=0$. 
Now, by Corollary~\ref{idekr}, Theorem~\ref{rega} is proved.
\qed

\section{EKR-module property for primitive rank 3 group actions.}\label{rank3}

In this section, we study the EKR-module property for primitive permutation groups of rank 3, and prove Theorem~\ref{Rank3}.
Let $G$ be a primitive permutation group on $\Ome$ of rank 3.
To prove Theorem~\ref{Rank3}, we may assume that $G$ is not an almost simple group.
Then either
\begin{enumerate}[(a)]
\item $G$ is affine, so that $G$ has a regular normal subgroup, or
\item $G$ is in product action, and $G\leqslant T \wr S_{2}$ on $\Omega^2$, where $T \leqslant \Sym(\Omega)$ is $2$-transitive.
\end{enumerate}

If $G$ is affine, then $G$ indeed has the EKR-module property by Theorem~\ref{rega}.
We thus assume further that $G$ is in product action in the rest of this section.

Let $T \leqslant \Sym(\Omega)$ be a $2$-transitive group, and let $H=T_{\omega} < T$, where $\omega \in \Omega$.
Let $G= T \wr S_{2}$, and $M=H \wr S_{2}$.
Then $G$ naturally acts on $\Omega^2$, with $M=G_{(\omega, \omega)}$. 
Obviously, 
\[\mbox{$\{(\omega,\ \omega)\}$, $\left(\Omega\setminus \{\omega\}\right) \times \left(\Omega\setminus \{\omega\}\right)$, and $ \left(\left(\Omega\setminus \{\omega\}\right) \times \{\omega\} \right) \bigcup \left(\{\omega\} \times \left(\Omega\setminus \{\omega\}\right) \right)$}\] 
are the orbits of $M=G_{(\omega, \omega)}$ on $\Omega^{2}$. 
Thus $G$ is of rank $3$. 

In view of Corollary~\ref{idekr}, it is beneficial to obtain descriptions of the set $\Irr(G)$ of irreducible characters of $G$, and of the maximum intersecting sets in $G$ containing the identity. As one would expect, the $2$-transitive action $T$ on $\Omega$ plays a major role.  Before going any further, we establish some notation. In $G=T \wr S_2=(T\times T)\rtimes S_2$, by $\pi$, we denote the unique $2$-cycle in  $S_{2}$. Elements of $G\setminus (T\times T)$ are of the form $(s,r) \pi$, where $s,r \in T$. By $(s,r)\pi$, we denote the product of elements $(s,r)$ and $\pi$ of $G$.

We start by describing $\Irr(G)$. The subgroup $N:=T\times T$ of $G$ is a normal subgroup of index $2$. By Clifford theory (\cite[6.19]{isaacs1994character}), restriction of any irreducible character $\nu \in \Irr(G)$ to $N$ is either an irreducible $G$-invariant character of $N$, or the sum of two $G$-conjugate irreducible characters of $N$. From well-known results on characters of direct products, we have
$$\Irr(N)=\{\chi \times \lambda \ : \chi,\lambda\in \ \Irr(T)\}.$$ Let $\chi, \lambda$ be two distinct irreducible characters of $T$, then the inertia subgroup in $G$ of $\chi \times \lambda $ is $N$, and therefore $\sigma_{\chi,\lambda }:=Ind^{G}_{N} (\chi \times \lambda)$ is an irreducible character of $G$, with $Res^{G}_{N}(\sigma_{\chi,\lambda })=\chi \times \lambda + \lambda \times \chi$. Now consider an irreducible character of $N$, of the form $\chi\times \chi$. Let $P : T \to \mrm{GL}(V)$ be a representation affording $\chi$ as its character. Then $P\otimes P : N \to \mrm{GL}(V \otimes V)$ is a representation of $N$ that affords $\chi \times \chi$ as its character. Let $\pi$ be the unique $2$-cycle in $S_{2}$.
Define $\Psi : G \to \mrm{GL}(V\otimes V)$ to be the representation such that $\Psi\lvert_{N}=P\otimes P$ and $\Psi(\pi)(v \otimes w)= w\otimes v$ for all $v,w \in V$. The character $\rho_{\chi}$ afforded by $\Psi$ is an irreducible character of $G$ that extends $\chi\times \chi$. We also have $\rho_{\chi}((s,r)\pi))=\chi(rs)$ for all $r,s \in T$. By a result of Gallagher (\cite[6.17]{isaacs1994character}), there is exactly one other irreducible character of $G$ whose restriction to $N$ is $\chi\times \chi$, namely $\beta \rho_{\chi}$, where $\beta$ is the unique non-trivial linear character with kernel $N$. Therefore by  Clifford theory any irreducible character is one of the characters defined above.

\begin{lemma}
The set
$$\{\rho_{\chi}\ :\ \chi  \in \Irr(T)\}\cup \{\beta\rho_{\chi}\ :\ \chi  \in \Irr(T)\}\cup \{\sigma_{\chi,\ \lambda}\ :\ \chi,\lambda \in \Irr(T)\ \&\ \chi \neq \lambda\}$$ is the complete set of irreducible characters of $G$.
\end{lemma}

We now describe the permutation character for the action $G$ on $\Omega^{2}$. As $T$ is a $2$-transitive group, there is $\psi \in \Irr(T)$ be such that $1+\psi$ is the permutation character for $T$. 
Computation shows that $\Lambda := 1 +\rho_{\psi}+\sigma_{\psi,1}$ is the permutation character for $G$.

The next lemma follows from the proof of Lemma~3.5 of \cite{hujdurovic2022maximum}, which is essentially the same as Lemma~4.2 of \cite{AM2015}.

\begin{lemma}\label{directproduct}
Every maximum intersecting set for the action of $T\times T$ on $\Omega^{2}$ is of the form $S \times R$, where $S$ and $R$ are maximum intersecting sets with respect to the action of $T$ on $\Omega$
\end{lemma}

We now give the following characterization of maximum intersecting sets in $G=T\wr S_2$.

\begin{lemma}\label{mispa}
The action of $G$ on $\Omega^{2}$ satisfies the EKR property. If $\I$ is a maximum intersecting set in $G$ that contains the identity, then there are maximum intersecting sets $X$, $W$, $Z$ in $T$ such that:
\begin{enumerate}[{\rm(i)}]
\item $\I= (W \times Z) \cup (X\times X^{-1})\pi$, and

\item $W$ and $Z$ contain the identity of $T$.
\end{enumerate}
\end{lemma}
\begin{proof}
As $T$ is a $2$-transitive group, by the main results of \cite{MST2016} and \cite{MSi2019}, the action of $T$ on $\Omega$ satisfies both EKR and EKR-module properties. By Lemma~\ref{directproduct}, a maximum intersecting set for the action of $N:=T\times T$ on $\Omega^{2}$ is of the form $S\times R$, where $S$ and $R$ are maximum intersecting sets in $T$. Therefore, the action of $N$ on $\Omega^{2}$ also satisfies the EKR property.
The subgroup $N$ of $G$ is a transitive subgroup satisfying the EKR property, and so by Lemma~3.3 of \cite{MST2016}, we see that the action of $G$ also satisfies the EKR property.

We consider a maximum intersecting set $\I$ with respect to the action of $G$ on $\Omega^{2}$. We further assume that $\I$ contains the identity element. With this assumption, every element of $\I$ must fix a point in $\Omega^{2}$.
Now $\I \cap N$ and $(\I \cap N\pi)\pi^{-1}$ are intersecting sets with respect to the action of $N$ on $\Ome^{2}$. We note that $H \times H \leq N$ is a point stabilizer for this action. Since the action of $N$ on $\Ome^{2}$ satisfies the EKR property, we have
$|H\times H| \geq |(\I \cap N)|$ and  $
|H \times H| \geq |(\I \cap N\pi)\pi^{-1}|$. Now since the action of $G$ on $\Ome^{2}$ satisfies the EKR property, $M$ is a point stabilizer, and $\I$ is a maximum intersecting set in $G$, we have
$2|H \times H|=|M|=|\I|=|(\I \cap N\pi)\pi^{-1}|+|(\I \cap N)|$.
Therefore, both $\I \cap N$ and $(\I \cap N\pi)\pi^{-1}$ are maximum intersecting sets in $N$. Using Lemma~\ref{directproduct}, we see that there are maximum intersecting sets $W$, $Z$, $X$, $Y$ in $T$, such that (i) $\I \cap N=W\times Z$; and (ii) $(\I \cap N\pi)=(X\times Y)\pi$. As $\I \cap N$ contains the identity of $N$, $W$ and $Z$ must contain the identity of $T$. We will now show that $X^{-1}=Y$.

Given $x\in X$ and $y \in Y$, consider the element $(x,y)\pi \in (\I \cap N\pi) \subset \I$. 
As we assume that $\I$ contains the identity, $(x,y)\pi$ must fix a point. 
That is to say, $0\neq \Lambda((x,y)\pi)=1+\psi(xy)$, where $\Lambda$ and $\psi$ are as described prior to the statement of the lemma. As $1+\psi$ is the permutation character for the action of $T$ on $ \Ome$, $1+\psi(xy)\neq 0$ if and only if $xy\in T$ fixes a point of $\Ome$. Thus for for a given $y\in Y$, the set $X\cup \{y^{-1}\}$ is an intersecting set in $T$. As $X$ is a maximum intersecting set in $T$, we must have $y^{-1} \in X$. This shows that $Y=X^{-1}$.
\end{proof}

We recall that $\Lambda=1 +\rho_{\psi}+\sigma_{\psi,1}$ is the permutation character for the action of $G$ on $\Omega^2$, where $\psi\in \Irr(T)$ is such that $1+\psi$ is the permutation character for the action of $T$ on $\Omega$.
By Corollary~\ref{idekr}, EKR-module property of $G$ is equivalent to showing that $\nu(\ch_{\I})=0$ for all maximum intersecting sets $\I$ that contain the identity and $\nu \in \Irr(G) \setminus \{1, \sigma_{\psi,1},\  \rho_{\psi}\}$.
Let $\I_{0}$ be a maximum intersecting set in $G$ such that $1_{G} \in \I_{0}$. By Lemma~\ref{mispa}, there are maximum intersecting sets $X,W,Z$ in $T$ such that : $Z$ and $W$ contain the identity of $T$; and $\I_{0}= W\times Z \cup \left(Z \times Z^{-1} \right) \pi$. For any distinct pair $\chi, \lambda \in \Irr(T)$, we can compute the following character sums:
\begin{enumerate}[(I)]
\item $\rho_{\chi}(\ch_{\I_{0}})=\chi(\ch_{W})\chi(\ch_{Z})+\sum\limits_{r,s \in Z}\chi(r^{-1}s)=\chi(\ch_{W})\chi(\ch_{Z})+\sum\limits_{r\in Z}\chi(\ch_{r^{-1}Z})$;

\item  $\beta\rho_{\chi}(\ch_{\I_{0}})=\chi(\ch_{W})\chi(\ch_{Z})-\sum\limits_{r,s \in Z}\chi(r^{-1}s)=\chi(\ch_{W})\chi(\ch_{Z})-\sum\limits_{r\in Z}\chi(\ch_{r^{-1}Z})$; and

\item $\sigma_{\chi,\ \lambda}(\ch_{\I_{0}})=\chi(\ch_{W})\lambda(\ch_{Z})+\lambda(\ch_{W})\chi(\ch_{Z})$.
\end{enumerate}
We need to compute $\chi(\ch_{S})$, for all $\chi \in \Irr(T)$ and all maximum intersecting sets $S$ in $T$. To do so, we use the EKR and EKR-module properties of $2$-transitive groups.

\begin{lemma}\label{2tra}
Let $T\leq \Sym(\Omega)$ be a $2$-transitive group with $H\leq G$ being the point stabilizer. 
Let $\psi\in \Irr(T)$ be such that $1+\psi$ is the permutation character. If $S$ is an maximum intersecting set in $T$, then
\begin{enumerate}[{\rm(i)}]
\item $\psi(S)=|H|$, when $1\in S$;
\item $\psi(S)=-|H|/\psi(1)$, provided $1\notin S$; and 
\item $\nu(S)=0$ for all irreducible characters $\nu\notin\ \{1,\ \psi\}$.
\end{enumerate}
\end{lemma}

\begin{proof}
As $T$ is 2-transitive, it satisfies both the EKR and EKR-module properties. 
By Corollary~\ref{2t}, $\ch_{\I}$ is in the ideal $J=\left\langle e_{1} \right\rangle+\left\langle e_{\psi} \right\rangle$. By the orthogonality relations among primitive central idempotents, we see that left multiplication by $e_{1}+e_{\psi}$ is a projection onto $J$. That is $(e_{1}+e_{\psi})(\ch_{\I})=\ch_{\I}$. Writing both sides as a linear combination of the elements in the basis set $\{t\in T\}$ of the group algebra $\bb{C}T$, and equating the coefficients of the identity element on both sides, yields the first two formulae.

Part (iii) is a direct consequence of Corollary~\ref{idekr}.
\end{proof}

Pick a maximum intersecting set $\I_{0}$ in $G$ such that $1_{G} \in \I_{0}$, and let $\nu \in \Irr(G)\setminus \{1,\ \sigma_{\psi,1},\ \rho_{\psi},\ \beta\rho_{\psi}\}$. Now applying Lemma~\ref{2tra}\,(ii) and the character sum formulas (I) (II) and (III) given above yields that $\nu(\ch_{\I_{0}})=0$. 
As $1,\ \sigma_{\psi,1},\ \rho_{\psi}$ are the only irreducibles that contribute to the permutation character for the action of $G$ on $\Ome^{2}$, in view of Corollary~\ref{idekr}, we need to show that
$\beta\rho_{\psi}(\ch_{\I_{0}}) = 0$. 
This is indeed true by Lemma~\ref{2tra}, and then Theorem~\ref{Rank3} follows from an application of Corollary~\ref{idekr}.
\qed

\section{Some groups satisfying the EKR-module property.}

In this section, we study groups satisfying the EKR-module property. 
Recall (from Definition~\ref{sekrm}) that a finite group $G$ satisfies the EKR-module property if every transitive action of $G$ satisfies the EKR-module property. 
We first prove Theorem~\ref{Nil}, and then prove the smallest non-abelian simple group $A_5$ satisfies the EKR-module property.

\subsection{Proof of Theorem~\ref{Nil}.}\label{class2}

In this subsection, we consider transitive actions of nilpotent groups of nilpotency class $2$. By \cite[Theorem 3]{BMMKK2015}, all transitive actions of nilpotent groups satisfy the EKR property. In the same paper, it was also shown that there are examples of class-$2$ nilpotent groups that do not satisfy the strict-EKR property. We will show that all transitive actions of class-$2$ nilpotent groups satisfy the EKR-module property.
Our proof is a proof by contradiction.

Recall the following well-known result from character theory.

\begin{lemma}\label{central}
Let $G$ be a group, $\psi$ an irreducible complex character of $G$, and $z$ an element of the centre of $G$. Then for all $g\in G$, we have $\psi(gz)=\psi(g)\psi(z)$.
\end{lemma}
\begin{proof}
Let $\rho: G \to \mrm{GL}(V)$ be a representation affording $\psi$ as its character. As $z$ is in the centre, the map $\rho(z): V\to V$ is a $G$-module homomorphism. Thus, by Schur's lemma, $\rho(z)$ acts like a scalar matrix, and thus $\psi(gz)=Tr(\rho(gz))=Tr(\rho(g))Tr(\rho(z))=\psi(g)\psi(z)$.
\end{proof}

Assume that Theorem~\ref{Nil} is false. 
Let $N$ be a class-$2$ nilpotent group $N$, and $H\leqslant N$ such that the action of $N$ on $\Omega=[N:H]$ does not satisfy the EKR-module property.
We may further assume that $|N|+|\Omega|$ is as small as possible. By the minimality of $(N,\Omega)$ and by Corollary~\ref{permred}, the action of $N$ on $\Omega$ must be a permutation action. In other words, $H$ is core-free, that is, $\bigcap\limits_{n \in N} nHn^{-1}=\{1_{N}\}$. As the action of $N$ on $\Omega$ does not satisfy the EKR-module property, by Theorem~\ref{idekrm}, there is a character $\chi \in \{\psi:\ \psi \in \Irr(N)\ \&\ \psi(\ch_{H})=0\}$ and a maximum intersecting set $\I$ such that $\chi(\ch_{\I})\neq 0$. We fix one such pair $\chi, \I$.

As $N$ is nilpotent, it has a non-trivial centre, which we denote by $Z$. Given a character $\psi$ of $N$, we denote its kernel, $\{n\in N: \ \psi(n)=\psi(1)\}$, by $ker(\psi)$.
As every non-trivial normal subgroup of a nilpotent group intersects non-trivially with the centre, we either have $ker(\chi) \cap Z \neq \{1_{N}\}$, or that $\chi$ is a faithful character.

We first assume that $\chi$ is faithful. As $N$ is a class-$2$ nilpotent group, we have $Z \supset [N,N]$. Since $N$ is non-abelian, given  $y \in N \setminus Z$, we can pick $x \in N$ be such that $z:=xyx^{-1}y^{-1} \neq 1_{N}$. As $\chi$ is faithful, we have $\chi(z)\neq 1$. Since $xyx^{-1}=zy$, we have $\chi(y)=\chi(xyx^{-1})=\chi(yz)$. As $z$ is a central element, by Lemma~\ref{central}, we have $\chi(y)=\chi(y)\chi(z)$. As $\chi(z)\neq 1$, we must have $\chi(y)=0$ for all $y \in N\setminus Z$. Recall that $H$ is core-free, and thus $H \cap Z =\{1_{N}\}$. We can now conclude that $\chi(\ch_{H})=\chi(1)\neq 0$. This contradicts our initial condition that $\chi(\ch_{H})=0$, and therefore $\chi$ cannot be faithful.

Now we are left with the case when $\chi$ is not faithful. 
By $ker(\chi)$, we denote the kernel of a corresponding representation. 
We set $Z_{\chi}=ker(\chi) \cap Z$. 
We note that $Z_{\chi}$ is a non-trivial normal subgroup of $N$. 
As $H$ is a core-free subgroup, we have $Z_{\chi} \cap nHn^{-1}=\{1_{N}\}$, for all $n\in N$. 
Thus the action of $Z_{\chi}$ on $\Omega$ is semi-regular. 
If the action of $Z_{\chi}$ is regular, then it is a regular normal subgroup, and thus by Theorem~\ref{rega}, the action of $N$ on $\Omega$ must satisfy the EKR-module property. As this contradicts our assumption, the action of $Z_{\chi}$ on $\Ome$ must be semi-regular and intransitive. 
As $Z_{\chi}\lhd N$, the set $\tilde{\Ome}$ of $Z_{\chi}$ orbits on $\Ome$, is a block system for the action of $N$ on $\Ome$. 
Since $Z_{\chi}$ acts intransitively, we have $|\tilde{\Ome}| \lneqq |\Ome|$, and thus $|N|+|\tilde{\Ome}|\lneqq |N|+|\Ome|$. We now consider the transitive action of $N$ on $\tilde{\Ome}$. The elements of $HZ_{\chi}$ fix the $Z_{\chi}$-orbit containing the element $H \in \Ome$. Observing that $|N|/|HZ_{\chi}|=|N|/|H||Z_{\chi}|=|\Ome|/|Z_{\chi}|=|\tilde{\Ome}|$, we can conclude that $HZ_{\chi}$ is a stabilizer for the action of $N$ on $\tilde{\Ome}$. As  $\I$ is an intersecting set with respect to the action of $N$ on $\Ome$, the set $\I Z_{\chi}$ is an intersecting set with respect to the action of $N$ on $\tilde{\Ome}$. Since $Z_{\chi}$ is a central semi-regular subgroup in $N \leq \Sym(\Ome)$ and  $\I$ is an intersecting set, we can conclude that $|\I Z_{\chi}|=|\I||Z_{\chi}|$. As we mentioned above, transitive actions of nilpotent groups satisfy the EKR property, and thus since $\I$ is a maximum intersecting set with respect to the action of $N$ on $\Ome$, we have $|\I|=|H|$, and therefore $|\I Z_{\chi}|=|H Z_{\chi}|$. We can now see that $\I Z_{\chi}$ is a maximum intersecting set with respect to the action of $N$ on $\tilde{\Ome}$. As $Z_{\chi}\leq ker(\chi)$ is a central subgroup, by using Lemma~\ref{central}, we have $\chi(\ch_{\I Z_{\chi}})=|Z_{\chi}| \times \chi(\ch_{\I})$ and $\chi(\ch_{H Z_{\chi}})=|Z_{\chi}| \times \chi(\ch_{H})$. By our choice of $\chi$ and $\I$, $\chi(\ch_{H})=0$ and $\chi(\ch_{\I})\neq 0$. Therefore $\I Z_{\chi}$ is a maximum intersecting set with respect to the action of $N$ on $\tilde{\Ome}$ and $\chi$ is a character in $\{\psi \in \Irr(G)\ :\ \psi(\ch_{H Z_{\chi}})=0\}$, such that $\chi(\ch_{\I Z_{\chi}})\neq 0$. So by Theorem~\ref{idekrm}, the action of $N$ on $\tilde{\Ome}$ does not satisfy the EKR-module property. Now since $|N|+|\tilde{\Ome}|\lneqq |N|+|\Ome|$, this conclusion contradicts the minimality of $(N,\ \Ome)$. Therefore our assumption that $\chi$ is not faithful must be false.

Both cases return contradictions, and hence our initial assumption that Theorem \ref{Nil} fails, must be false. This concludes the proof.
\qed

Theorem~\ref{Nil} and Theorem 3 of \cite{BMMKK2015} establish the existence of infinitely many groups that satisfy the EKR and EKR-module property, but not the strict-EKR property. Theorem 2 of \cite{BMMKK2015} shows that groups that satisfy the EKR property are necessarily solvable. However, the EKR-module property is not so restrictive.

\subsection{A group satisfying the EKR-module property is not necessarily solvable}\ 


\begin{lemma}\label{A5}
The simple group $A_{5}$ satisfies the EKR-module property.
\end{lemma}
\begin{proof}
Let $H$ be a subgroup of $A_{5}$. 
We need to show that the action of $A_{5}$ on $\Omega_{H}=[A_{5}:H]$ satisfies the EKR-module property. 
Assume that $H$ is a subgroup satisfying \begin{equation}\label{trivial}
\{\chi \in \Irr(A_{5})\ :\ \chi(\ch_{H})\neq 0\}=\Irr(A_{5}).
\end{equation} 
Then by Theorem~\ref{idekrm}, the action of $A_{5}$ on $\Omega_{H}$ satisfies the EKR-module property. 
Computation shows that the relation \eqref{trivial} fails if and only if $H$ is isomorphic to one of the groups:
$$\bb{Z}^{2}_{2},\ \bb{Z}_{5},\ S_{3},\ D_{10},\ A_{4}.$$ 
(Here $D_{10}$ denotes the dihedral group of order 10.) 
We will deal with groups separately.

When $H$ is isomorphic to one of $D_{10}$ or $A_{4}$, the action of $G$ on $\Omega_{H}$ is $2$-transitive. Hence by the main result of \cite{MSi2019}, these group actions satisfy the EKR-module property.

Consider a subgroup $H_{1} \cong \bb{Z}_{5}$. We will use Theorem~\ref{Specsuff} to show that the action of $A_{5}$ on $\Omega_{H_{1}}$ satisfies the EKR-module property. For this, we need the character table of $A_{5}$, which is given as Table~\ref{ch1}. 
\begin{table}[h]
\caption{Character table of $A_{5}$.}
\centering
\begin{tabular}{c | c c c c c}
  class&$(\ )$&$C_{1}$&$C_{2}$&$C_{3}$&$C_{4}$\\
  size&$1$&$15$&$20$&$12$&$12$\\
\hline
  $\rho_{1}$&$1$&$1$&$1$&$1$&$1$\\
  $\rho_{2}$&$3$&$-1$&$0$&$\frac{1+\sqrt{5}}{2}$&$\frac{1-\sqrt{5}}{2}$\\
  $\rho_{3}$&$3$&$-1$&$0$&$\frac{1-\sqrt{5}}{2}$&$\frac{1+\sqrt{5}}{2}$\\
  $\rho_{4}$&$4$&$0$&$1$&$-1$&$-1$\\
  $\rho_{5}$&$5$&$1$&$-1$&$0$&$0$
\end{tabular}
\label{ch1}
\end{table}
In this case, the set $Der(G, \Omega_{H_1})$, of derangements, is the union of the conjugacy class $C_{1}$ containing $(1,2)(3,4)$ and the conjugacy class $C_{2}$ containing $(1,2,3)$. Let $f_{1}$ be the $(G,\Omega_{H_{1}})$-compatible class function satisfying $f_{1}((1,2)(3,4))=1$ and $f_{1}((1,2,3))=2$. Now application of Theorem~\ref{Specsuff} by setting $f=f_{1}$ and $\I=H_{1}$, yields that this action satisfies the EKR-module property.

Next, we consider a subgroup $H_{2} \cong \bb{Z}^{2}_{2}$ in $A_{5}$ and the action of $A_{5}$ on $\Omega_{H_{2}}$. The set $Der(G, \Omega_{H_{2}})$, of derangements, is the union of the conjugacy class $C_{2}$ containing $(1,2,3)$, the conjugacy class $C_{3}$ containing $(1,2,3,4,5)$, and the conjugacy class $C_{4}$ containing $(1,5,4,3,2,1)$. Let $f_{2}$ be be the $(G,\Omega_{H_{2}})$-compatible class function satisfying $f_{2}((1,2,3))=1$ and $f_{2}((1,2,3,4,5))=f_{2}((1,5,4,3,2,1))=3/2$. Now application of Theorem~\ref{Specsuff}, by setting $\I=H_{2}$ and $f=f_{2}$, yields that this action satisfies the EKR-module property.

Finally, we consider a subgroup $H_{3}\cong S_{3}$ in $A_{5}$ and the action of $A_{5}$ on $\Omega_{H_{3}}$. The set $Der(G, \Omega_{H_{2}})$, of derangements, is the union of conjugacy classes $C_{3}$ and $C_{4}$. Let $f_{3}$ be be the $(G,\Omega_{H_{3}})$-compatible class function satisfying $f_{3}((1,2,3,4,5))=f_{3}((1,5,4,3,2,1))=1$. Let $K\cong A_{4}$ be a subgroup of $A_{5}$. As $KK^{-1}=K \subset \bigcup\limits_{g \in G} H_{3}^{g}$, we see that $K$ is an intersecting set with respect to this action. 
Now, setting $\I=K$ and $f=f_{3}$, Theorem~\ref{Specsuff} yields that this action satisfies the EKR-module property.
\end{proof}

\section{EKR-module property in Strongly Regular Graphs}\label{srg}

In the section, we consider maximum cliques in the Peisert-type strongly regular graphs defined in \cite{EKRpaley}. These are a subclass of strongly regular graphs found in \cite{brouwer1999cyclotomy}. Consider a strongly regular graph $X$, with a prescribed set $\mathcal{C}$ of ``naturally'' occurring cliques. Cliques in $\mathcal{C}$ will be called canonical cliques. We say that $X$ satisfies the EKR-module property with respect to $\mathcal{C}$ if the characteristic vector of every maximum clique in $X$ is a linear combination of characteristic vectors of the cliques in $\mathcal{C}$. We now define Peisert-type graphs.

\begin{definition}
{\rm
Let $q$ be an odd prime power. Then a {\it Peisert-type graph of type $(m,q)$} is a Cayley graph on the additive group of $\bb{F}_{q^{2}}$ with its ``connection'' set $S$ being a union of $m$ cosets of $\bb{F}^{\times}_{q}$ in $\bb{F}^{\times}_{q^{2}}$ such that $\bb{F}^{\times}_{q} \subset S$.
}
\end{definition}

Given a Peisert-type graph of type $(m,q)$, with connection set $S$. For any $s \in S$ and $x\in \bb{F}_{q^{2}}$, the set $s\bb{F}_{q}+x$ is a naturally occurring clique. By a canonical clique in a Peiset-type graph, we mean a clique of the form $s\bb{F}_{q}+x$, where $s \in S$ and $x \in \bb{F}_{q^{2}}$. The main result of \cite{EKRpaley} is the following shows that Peiser-type graphs satisfy EKR-module property. In this section, we give a shorter and different proof of the same.

\begin{theorem}(\cite[Theorem 1.3]{EKRpaley})\label{EKRP}
The characteristic vector of a maximum clique in a Peisert-type graph is a linear combination of characteristic vectors of its canonical cliques.
\end{theorem}

We now collect some results about Peisert-type graphs and some general results about strongly regular graphs.
The main result of \cite{brouwer1999cyclotomy} shows that Peisert-type graphs are strongly regular. A different proof of the same is given in \cite{EKRpaley}. We will give a proof using a standard technique of finding eigenvalues of an abelian Cayley graph.

\begin{lemma}\label{par}
Peisert-type graph of type $(m,q)$ is strongly regular with eigenvalues $k:=m(q-1)$ with multiplicity 1, $q-m$ with multiplicity $m(q-1)$, and $-m$ with multiplicity $q^{2}-1-m(q-1)$.
\end{lemma}
\begin{proof}
Let $X$ a Peisert-type graph of type $(m,q)$ whose connection set is $S=\bigcup\limits_{i=0}^{m-1} c_{i}\bb{F}^{\times}_{q}$ (with $c_{0}=1$).
Let $A$ be the adjacency matrix of $X$. Considering an additive character of $\chi$ of $\bb{F}_{q}$ as a column vector, we see that $A\chi=\chi(\ch_{S})\chi$, where $\chi(\ch_{S})=\sum\limits_{s \in S}\chi(s)$.

If $\chi$ is not the trivial character, $Ker(\chi) \neq \bb{F}_{q^{2}}$, and so at most one of $\{c_{i}\bb{F}_{q}\ :\ 0\leq i \leq m-1\}$ can be a subgroup of $Ker(\chi)$. Thus if $c_{i}\bb{F}_{q} \subset Ker(\chi)$ for some $i$, then the restriction $\chi|_{c_{j}\bb{F}_{q}}$ of $\chi$ onto the subgroup $c_{j}\bb{F}_{q}$, is a non-trivial character whenever $j \neq i$. Otherwise, $Ker(\chi)$ will have two $1$-dimensional subspaces of $\bb{F}_{q^{2}}$ and thus must be equal to $\bb{F}_{q^2}$.

Assume that $\chi $ is a non-trivial character with $c_{i}\bb{F}_{q} \subseteq Ker(\chi)$. As the sum of values of a non-trivial character are zero, in this case, we have
\begin{align*}
\chi(\ch_{S}) &= \sum\limits_{x \in c_{i}\bb{F}^{\times}_{q}} \chi(x) + \sum\limits_{j \neq i} \sum\limits_{x \in c_{j}\bb{F}^{\times}_{q}} \chi(x) \\
&= q-1 -(m-1)=q-m.
\end{align*}
The set on non-trivial characters $\chi$ with $c_{i}\bb{F}_{q} \subset Ker(\chi)$ is in one-one correspondence with the non-trivial characters of $\bb{F}_{q^{2}}/c_{i}\bb{F}_{q}$. Thus there are atleast $m(q-1)$ characters $\chi$ such that $A\chi=(q-m)\chi$. As distinct characters are orthogonal the dimension of the $(q-m)$-eigenspace of $A$ is atleast $m(q-1)$.

Similarly, if $\chi$ is a non-trivial character with $c_{i}\bb{F}_{q} \not\subseteq Ker(\chi)$ for all $0\leq i \leq m-1$, we have $\chi(\ch_{S})=-m$. Thus there are atleast $q^{2}-1-m(q-1)$ characters $\chi$ such that $A\chi=(-m)\chi$. As distinct characters are orthogonal the dimension of the $(-m)$-eigenspace of $A$ is atleast $q^{2}-1-m(q-1)$.

If $\chi$ is the trivial character $\chi_{0}$, we have $A\chi_{0}=|S|\chi_{0}$. With this, we have found all the eigenvalues of $A$ and their corresponding eigenspaces. As $A$ has exactly three distinct eigenvalues, it is a strongly regular graph.
\end{proof}

Let $X$ be a strongly regular graph with parameters $(v,\ k,\ \lambda,\ \mu)$, which is a $k$-regular graph on $v$ vertices such that 
\begin{enumerate}[(i)]
\item  any two adjacent vertices have exactly $\lambda$ common neighbours, and 
\item any two non-adjacent vertices have exactly $\mu$ common neighbours. 
\end{enumerate}
We further assume that $X$ is primitive, that is, both $X$ and its complement are connected. It is well known (\cite[Lemma 10.2.1]{GR2001}) that the adjacency matrix $A$ of $X$ has exactly three distinct eigenvalue. As $X$ is $k$-regular and connected, $k$ is an eigenvalue of $A$ with multiplicity $1$. Let $r,s$ with $r>s$ be the other eigenvalues.

Our proof uses some results on graphs in association schemes. For a quick introduction to the preliminaries on graphs in association schemes, we refer the reader to Chapter 3 of \cite{GMbook}. 
We first recall the following well-known result linking strongly regular graphs with association schemes.
By $J$ and $I$, we denote the all-one matrix and the identity matrix respectively.

\begin{lemma}(\cite[Lemma 5.1.1]{GMbook}) Let $X$ be a graph with $A$ as its adjacency matrix. Then $X$ is strongly regular if and only if $\mathcal{A}_{X}:=\{I,\ A,\ \ov{A}:=J-I-A\}$ is an association scheme.
\end{lemma}

By $\bb{C}[\mathcal{A}_{X}]$, we denote the linear span of matrices in $\mathcal{A}_{X}$. This is referred to as the Bose-Mesner algebra. By a well-known result (\cite[Theorem 3.4.4]{GMbook}), the projections onto eigenspaces of $A$ is an orthogonal basis of idempotents of $\bb{C}[\mathcal{A}_{X}]$. The matrix $J$ is the projection onto the $k$-eigenspace. We denote $E_{r}$ and $E_{s}$ to be the projections onto the $r$-eigenspace and the $s$-eigenspace respectively. We have $A=kJ+rE_{r}+sE_{s}$, $I=\frac{J}{n}+E_{r}+E_{s}$, and so $\{\frac{J}{n},\ E_{r},\ E_{s}\}$ is an orthogonal basis of idempotents of $\bb{C}[\mathcal{A}_{X}]$. We now mention a bound by Delsarte (see equation $(3.25)$ of \cite{delsarte1973algebraic}) on cliques in strongly regular graphs. We state the formulation of this result as given in \cite[Corollary 3.7.2]{GMbook}.

\begin{lemma}\label{godsilub}
Let $X$ be $k$-regular strongly regular graph with $s$ as the least eigenvalue of its adjacency matrix. If $C$ is a clique in $X$, then $|C| \leqslant 1-\frac{k}{s}$.

Moreover, if $C$ is a clique that meets the bound with equality, then the characteristic vector $\ch_{C}$ is orthogonal to the $s$-eigenspace.
\end{lemma}

Given a subset $B$ of the vertex set of $X$, by $\ch_{B}$, we denote the characteristic vector of $B$, and by $\mathbf{1}$, the all-one vector. Consider the $\bb{C}$-linear span $V_{max}$ of characteristic vectors of maximum cliques in $X$. By the above lemma, we have $|C| \leq 1-\frac{k}{s}$, for any clique $C$. Assume that there is a clique $C$ of size $1-\frac{k}{s}$, then by the above Lemma, $V_{max}$ is orthogonal to the $s$-eigenspace. We will now show that $V_{max}$ is in the image of $\frac{J}{n}+E_{r}$.

\begin{lemma}\label{evec}
Let $X$ be $k$-regular strongly regular graph on $n$ vertices with $\{k,\ r,\ s\}$ with $r>s$ as set of distinct eigenvalues of its adjacency matrix. If $X$ has a clique of size $1-\frac{k}{s}$, then $\ch_{C}-\frac{|C|}{n}\mathbf{1}$ is an $r$-eigenvector.
\end{lemma}
\begin{proof}
From Lemma~\ref{godsilub}, we have $E_{s}\ch_{S}=0$. As $\mathbf{1}$ is a $k$-eigenvector, it is also orthogonal to the $s$-eigenspace. Since $J \ch_{C} = |C| \mathbf{1}$, the vector $\ch_{C}-\frac{|C|}{n}\mathbf{1}$ is orthogonal to both the $k$-eigenspace and the $s$-eigenspace, and so must lie in the $r$-eigenspace.
\end{proof}

We are now ready to prove Theorem~\ref{EKRP}.
\vskip0.1in

{\bf Proof of Theorem~\ref{EKRP}:}\ 
Let $X$ be a Peisert-type graph of type $(m,q)$ whose connection set is $S=\bigcup\limits_{i=0}^{m-1} c_{i}\bb{F}^{\times}_{q}$ (with $c_{0}=1$). 
By Lemmas~\ref{par}, \ref{godsilub} and \ref{evec}, we obtain the next result.
\begin{lemma}
If $C$ is a maximum clique in $X$, then $|C|=q$ and $\ch_{C}-\frac{1}{q}\mathbf{1}$ is a $(q-m)$-eigenvector.
\end{lemma}

Thus, given $x \in \bb{F}_{q^{2}}$ and $0\leq i \leqslant m-1$, the canonical clique $c_{i}\bb{F}_{q}+x$ is a maximum clique and $\ch_{c_{i},x}:=\ch_{c_{i}\bb{F}_{q}+x}-\frac{1}{q}\mathbf{1}$ is a $(q-m)$-eigenevector. By Lemma~\ref{par}, the dimension of the $(q-m$)-eigenspace is $m(q-1)$. Using the above Lemma, we can now deduce that Theorem~\ref{EKRP} follows from showing that $\{\ch_{c_{i},x}\ : \ x \in \bb{F}_{q^{2}}\ \&\ 0\leqslant i \leqslant m-1\}$ spans an $m(q-1)$ dimensional vector space.

Given an additive character $\chi$ of $\bb{F}_{q^{2}}$, we set $\ch_{\chi}:=\sum\limits_{z\in \bb{F}_{q^{2}}}\chi(z)z$. 
In the proof of Lemma~\ref{par}, we see that the $(q-m)$-eigenspace $V_{q-m}$, is spanned by 
$$\{\ch_{\chi}: \ \text{ $\chi$ is non-trivial and $c_{i}\bb{F}_{q} \subseteq \ Ker(\chi)$ for some $0\leqslant i \leqslant m-1$}\}.$$
Let $E_{i}= \left\langle\{v_{\chi}\ :\ \bb{F}_{q^{2}}\supsetneq Ker(\chi) \supset c_{i}\bb{F}_{q}\} \right\rangle$. 
Then we have
$$V_{q-m}=\bigoplus E_{i}.$$

By $\left[\ ,\  \right]$, we denote the natural orthogonal form on $\bb{C}[\bb{F}_{q^{2}}]$. 
With respect to this form, we have $E^{\perp}_{i}= \bigoplus\limits_{j\neq i}E_{j}$. 
If $\chi$ is a non-trivial character, then 
$$\left[ \ch_{\chi}, \ch_{c_{i},x} \right] = \chi(x)\sum\limits_{z\in c_{i}\bb{F}_{q}}\chi(z).$$
Therefore, $\ch_{c_{i},x} \in E_{i}$ for all $x\in \bb{F}_{q^{2}}$. 
We will now show that vectors of the form $\ch_{c_{i},x}$ span $E_{i}$. 
Considering $x\in \bb{F}^{\times}_{q^{2}}$ and $f_{i,x}:= \ch_{c_{i},\ 0}- \ch_{c_{i},\ x}= \ch_{c_{i}\bb{F}_{q}}-\ch_{c_{i}\bb{F}_{q}+x}$. 
As the set $\{\ch_{c_{i}F_{q}+x}\ : \ x\in \bb{F}_{q^{2}}\}$ is a set of $q$ orthogonal vectors, the set $\{f_{i,x}: \ x\neq 0\}$ spans a $q-1$ dimensional vector space. 
Therefore $\{\ch_{c_{i},x}: \ x \in \bb{F}_{q^{2}}\}$ spans $E_{i}$. 
Thus $\{\ch_{c_{i},x}: \ x \in \bb{F}_{q^{2}}\ \&\ 0\leqslant i \leqslant m-1\}$ spans $V_{q-m}=\bigoplus E_{i}$. 
This concludes the proof of Theorem~\ref{EKRP}.
\qed

\bibliographystyle{plain}
\bibliography{ref}
\end{document}